\documentclass[a4paper]{jnsao}
\usepackage{amsmath,amsthm,enumerate}
 
\usepackage{amssymb}
\usepackage{amsfonts}
\usepackage{graphicx}
\usepackage{color}
\usepackage{makecell}
\usepackage{multirow}

\newtheorem{Theorem}{Theorem}[section]
\newtheorem{Example}{Example}[section]

\newtheorem{Lemma}{Lemma}[section]
\newtheorem{Remark}{Remark}[section]
\def\bc{\begin{center}}
\def\ec{\end{center}}

\def\s2c{\vskip 2cm}

\numberwithin{equation}{section}

\def\bt{\begin{Theorem}}
\def\et{\end{Theorem}}
\def\bl{\begin{Lemma}}
\def\el{\end{Lemma}}
\def\bcor{\begin{Corollary}}
\def\ecor{\end{Corollary}}

\begin{document}

\title{A simple algorithm for the simple bilevel programming (SBP) problem}

\author{Stephan Dempe\footnote{Faculty of Mathematics and Computer Science, TU Bergakademie Freiberg, Germany, e-mail: dempe@extern.tu-freiberg.de}, Joydeep Dutta\footnote{Department of Economic Sciences, Indian
Institute of Technology, Kanpur, India, e-mail: jdutta@iitk.ac.in}, Tanushree Pandit\footnote{Department of Mathematics, Cochin University of Science and Technology, Kerala, India, e-mail: tpandit@cusat.ac.in. Work of this author has been supported by the project SMNRI, CUSAT.} and K. S. Mallikarjuna Rao \footnote{Department of Industrial Engineering and Operations Research, Indian Institute of Technology, Bombay, India, e-mail: mallik.rao@iitb.ac.in}}
\maketitle

\begin{abstract}
In this article we intend to develop a simple and implementable algorithm for minimizing a convex function over the solution set of another convex optimization problem. Such a problem is often referred to as a simple bilevel programming (SBP) problem. One of the key features of our algorithm is that we make no assumption on the differentiability of the upper level objective, though we will assume that the lower level objective is smooth. Another key feature of the algorithm is that it does not assume that the lower level objective has a Lipschitz gradient, which is a standard assumption in most of the well-known algorithms for this class of problems. We present the convergence analysis and also some numerical experiments demonstrating the effectiveness of the algorithm.
\end{abstract}
Nonsmooth Optimization, Lipschitz Gradient.\\


\section{Introduction}
We consider the following Simple Bilevel Programming (SBP for short) problem.
\begin{equation}\label{sbp4}
\begin{cases}
\min f(x) \\
\mbox{subject to } x \in \mbox{argmin}\{ g(y) : y \in C \}.
\end{cases}
\end{equation}

\noindent We make the following assumptions throughout this paper.
\begin{enumerate}
\item[$\bullet$] The constrained set $ C $ is nonempty, compact and convex subset of $ \mathbb{R}^n $.

\item[$\bullet$] The upper level objective function $ f: \mathbb{R}^n \rightarrow \mathbb{R} $ and the lower level objective function $ g: \mathbb{R}^n \rightarrow \mathbb{R} $ are convex on $ \mathbb{R}^n $.

\item[$\bullet$] $g $ is assumed to be continuously differentiable on an open set containing $C$.
\end{enumerate}

This paper can be in some sense viewed as a continuation of the work carried out in \cite{Part1} and \cite{Part2}. In \cite{Part2}, an algorithm was developed where both the upper and lower level objective functions were assumed to be non-smooth. The algorithms described in a non-smooth setting involves computing the subgradients or $ \varepsilon $-subgradients, which is quite a hard job to do. So such algorithms are difficult to implement and can be implemented for very particular problems. However developing algorithms based on non-smooth data has its own advantage. We can get rid of the Lipschitz assumption on the gradient of the upper and lower level objective function; since we will not deal with the gradient in non-smooth setting. The next natural question that comes to us is that if we can get rid of the Lipschitz assumption on the gradients of the upper and lower level objective functions in a smooth (SBP) setup. In this article we develop an algorithm to fulfil this aim. Here we make no differentiability assumption on the upper level objective while the lower level objective function is considered to be differentiable but the gradient need not be Lipschitz continuous. We test our algorithm on several problems to demonstrate its efficiency.\\
Let us now discuss little bit about the algorithms present in the literature for smooth and non-smooth (SBP) problems. One of the earliest work in the area of implementable algorithm for the (SBP) problem is done by Solodov \cite{Solodov1}. In \cite{Solodov1}, he considered the upper and lower level objective functions of an (SBP) problem of the form (\ref{sbp4}) to be smooth with Lipschitz gradient. Solodov dealt with the non-smooth (SBP) problem in \cite{Solodov2} but the lower level problem was considered to be unconstrained. In that paper he used the bundle method to develop his algorithms for (SBP) problem. If the problem considered by Beck and Sabach \cite{Beck} is formulated as (SBP), then the objective function of the lower level problem is assumed to have a Lipschitz continuous gradient. The non-smooth (SBP) problem was also approached by Cabot \cite{Cabot}, where he has used subgradients to deal with the non-smoothness of the objective function. Given the $ k^{th} $ iteration $ x_k $, the next iteration $ x_{k+1} $ comes from the following expression
\begin{eqnarray}\label{opc}
- \frac{ x_{k+1} - x_k }{\lambda_k } \in \partial_{\varepsilon_k} ( g + \eta_k f ) ( x_{k+1} ),
\end{eqnarray}
where $ \varepsilon_k >0 $, $ \eta_k \rightarrow 0^+ $ and $ \lambda_k $ is a non-negative step size. Helou and Simoes \cite{Helou} has recently offered an algorithm for the non-smooth bilevel programming problem with constrained lower level problem which can be implemented for some specific applications by using a more relaxed version of the iteration scheme given by (\ref{opc}). Inspired by Cabot's idea, they have considered some intermediate step in the algorithm additional to Cabot's approach which helps in the computation. The algorithm they have presented includes optimality operators $ O_f $ and $ O_g $ corresponding to the objective functions $ f $ and $ g $ with certain properties. But the example for this kind of optimality operators is created with the help of subgradients of the objective functions which is again hard to compute in many cases.
So what we see in various papers is that either the objective functions are non-smooth or have Lipschitz gradient whenever it is smooth. The main reason behind this is that Lipschitz continuity of the objective function plays an important role in the convergence analysis of the algorithm. In this paper we have presented the numerical example \ref{example3}, where we can see that the gradient of the lower level objective function is Lipschitz continuous over the compact constraint set but not on whole $ \mathbb{R}^n $. In fact there are functions which do not have Lipschitz continuous gradient over a compact set also, which makes our assumption more relevant. Here we present the example.
\begin{Example}
    Consider the convex function $ g : \mathbb{R}^n \rightarrow \mathbb{R} $ such that
    \[
    g(x) = \begin{cases}
        x^{\frac{3}{2}}, \quad x\geq 0\\
        0, \quad x<0.
    \end{cases}
    \]
    Then $ g $ is continuously differentiable with
    \[
    g^{\prime}(x) = \begin{cases}
        \frac{3}{2}x^{\frac{1}{2}}, \quad x\geq 0\\
        0, \quad x<0.
    \end{cases}
    \]
    Now if we consider the constraint set as $ C = [-1,1]$, then $ g^{\prime} $ is not Lipschitz continuous over $ C $.
\end{Example}
The paper is organized as follows. In section 2, we provide the key motivation and ideas which has led to the development of the algorithm that we present here. In section 3, we present some lemmas and results already present in the literature which will be used in the convergence analysis later. In section 4, we present our algorithm as well as its convergence analysis. Finally in section 5 we present the results of numerical experimentation, showing the efficiency of our algorithm. To the best of our knowledge there does not seem to be any well-established set of test problems for (SBP). This may be due to the fact that the literature on algorithms for (SBP) has just started to emerge and only a handful of work in this direction has been done.

\section{General scheme for SBP} \label{gensch}
In this section, we will first present a general scheme to solve the simple bilevel programming problem (\ref{sbp4}). We plan to solve the following problem $( P_k )$ given by
\begin{equation*}
\begin{cases}
\min f(x) \\
\mbox{subject to } g(x) \leq \alpha_k \\
x \in C,
\end{cases}
\end{equation*}
for each $ k \in \mathbb{N} $. We expect that the sequence of the approximate solutions of the above problem will converge to the solution of (SBP) problem when $ \alpha_k \rightarrow \min\limits_C g $, as $ k \rightarrow \infty $. Thus comes the question as to why we do not solve the lower level problem of (\ref{sbp4}) first and then using the solution we solve the upper level problem. In other words, why don't we compute $ \min_C g = \alpha $ (say) and then solve the following problem
\begin{equation} \label{rsbp}
\begin{cases}
\min f(x)  \\
\mbox{subject to } g(x) \leq \alpha  \\
x \in C.
\end{cases}
\end{equation}
Note that the problem (\ref{rsbp}) does not satisfy the Slater condition, which makes it hard to solve. Following Beck and Sabach \cite{Beck} let us now explain the other reasons why this would not be a good idea to try to solve the problem (\ref{rsbp}) directly. When we solve an optimization problem, we hardly reach to the actual solution as one needs to terminate the sequence of iterates and thus stopping the algorithm. In most cases we are satisfied with an approximate solution upto a desired precision. The same way when we minimize the lower level objective function $ g $ over $ C $, we get $ \min_C g \pm \varepsilon $ where $ \varepsilon > 0 $ is a pre-decided measure of precision. Now note that for $ \alpha = \min_C g - \varepsilon $, the problem (\ref{rsbp}) is infeasible. Even for $ \alpha = \min_C g + \varepsilon $, solution of the problem (\ref{rsbp}) may not even provide any solution of the original (SBP) problem (\ref{sbp4}) and may not even be a good approximation. In fact through the following example we show that instead of using $ \alpha = \min_C $ in the constraints of (\ref{rsbp}) we use $ \alpha + \varepsilon $, which is an approximate optimal objective value of the lower level problem. By replacing $ \alpha $ with $ \alpha + \varepsilon $ in (\ref{rsbp}) and then solving (\ref{rsbp}) we will get a minimizer of (\ref{rsbp}) which is not the minimizer or even a good approximate minimizer of the original (SBP) problem.

\begin{Example}\textrm
Let us take $ \varepsilon > 0 $ be the pre-decided measure of precision. Now consider the following bilevel problem.
\begin{equation}\label{op}
\begin{cases}
 \min f(x) = (x-1)^2  \\
\mbox{subject to }  \\
x \in \mbox{argmin } \{ g(x) : x \in [-2, 2] \},
\end{cases}
\end{equation}
with
\begin{eqnarray} \label{g}
g(x) =
\begin{cases}
0 \quad & \mbox{if} \quad x \in ( - \infty , 0 ],\\
\delta x^2 \quad & \mbox{if} \quad x \in (0,\infty),
\end{cases}
\end{eqnarray}
where $ 0 < \delta \leq \varepsilon $.\\
Clearly, $ \mbox{argmin}_C g = [-2 , 0 ] $, $ \min_C g = 0 $ and $ \varepsilon $-solution set of the lower level problem {\textit i.e.} $ \{ x \in [-2, 2] : g(x) \leq \varepsilon \} $ contains the set $ [-2, 1] $. Which shows that if we consider the following problem
\begin{equation}\label{sp}
\begin{cases}
\min (x-1)^2  \\
\mbox{s. t. } g(x) \leq \varepsilon  \\
x \in [-2, 2],
\end{cases}
\end{equation}
the solution of (\ref{sp}) is $ \bar{x} = 1 $ with optimal value $ f(\bar{x}) = 0 $. It does not matter how small $ \varepsilon $ is but as long as $ \delta \leq \varepsilon $ and $ g $ is defined through (\ref{g}), (\ref{sp}) will have the same solutions. Whereas the solution of the original (SBP) problem (\ref{op}) is $ x^* = 0 $ with optimal value $ f(x^*) = 1 $. Thus clearly, the problem (\ref{sp}) in no way helps us to find the solution of problem (\ref{op}). \hfill $\Box$

\end{Example}
Here comes the designing of the problem $ (P_k) $, so that it will fulfill our goal. In order to solve $ (P_k) $ it will be helpful if the Slater condition holds. If $ \alpha_k > \min\limits_C g $, then by the definition of minimum we have that Slater condition holds for the problem $ (P_k) $ provided that $ \mbox{int } C \neq \emptyset $. Now all that is left is to design $ \alpha_k $ in a manner that it meets all our requirements. We have used the projected gradient method on $ g $ to do so and it will be discussed in details later.\\
The following theorem asserts that the solutions of these problems actually leads to the solution of the (SBP) mentioned above. In all the discussions that follow, we will denote the term $ \min\limits_C g $ by $ \alpha $, {\textit i.e.} we set $ \alpha = \min_C g $.

\begin{Theorem} \label{gen}
Let $ \{ \alpha_k \} $ is a sequence of real numbers such that $ \alpha_k > \alpha $ for all $ k \in \mathbb{N} $ and $ \lim\limits_{ k \rightarrow \infty } \alpha_k = \alpha $. Also let that $ x_k $ be a solution of the problem $ (P_k) $ mentioned above. Then any accumulation point of $ \{ x_k \} $ is a solution of the (SBP) problem.
\end{Theorem}
\proof Let $ \bar{x} $ is an accumulation point of the sequence $ \{ x_k \} $. Then there exists a subsequence $ \{ x_{k_j} \} $ such that $ \lim\limits_{ j \rightarrow \infty } x_{k_j} = \bar{x} $. As $ x_{k_j} $ is a feasible point of the problem $ (P_{k_j}) $ and $ C $ is closed, we have $ \bar{x} \in C $ and
\begin{eqnarray*}
g(x_{k_j}) \leq \alpha_{k_j}.
\end{eqnarray*}
Then by continuity of $ g $, $ g(\bar{x}) \leq \alpha $. This implies that $ \bar{x} \in \mbox{argmin}_C g $. By the assumption on $ \alpha_k $, for any $ \tilde{x} \in \mbox{argmin}_C g $, $ g(\tilde{x}) < \alpha_k $. Therefore $ \tilde{x} $ is a feasible point of $ (P_k) $ for all $ k \in \mathbb{N} $. Now as $ x_{k_j} $ is a solution of $ (P_{k_j}) $, we have
\begin{eqnarray*}
f( x_{k_j} ) \leq f(\tilde{x}) \quad \forall j \in \mathbb{N}.
\end{eqnarray*}
Then again by continuity of $ f $,
\begin{eqnarray*}
f(\bar{x}) \leq f( \tilde{x} ).
\end{eqnarray*}
Since this is true for any $ \tilde{x} \in \mbox{argmin}_C g $, we can conclude that $ \bar{x} $ is a solution of the (SBP) problem. \qed

\noindent Note that in the above theorem (Theorem \ref{gen} ) if we assume $ C $ to be a compact, convex set then for each $ k \in \mathbb{N}$, the problem $(P_k)$ has a solution and any accumulation point of the sequence of solutions of $(P_k)$ is a solution of the (SBP). Additionally if $ f : \mathbb{R}^n \rightarrow \mathbb{R} $ (or $ f : C \rightarrow \mathbb{R} $ ) be a strongly convex function then the sequence of solutions of $(P_k)$ converges to the unique solution of (SBP). We omit the simple proof.

\begin{Remark}
    After the proof of the above theorem it must be clear to the reader that we can actually revoke the compactness assumption on the constrained set $ C $ by assuming that the set $ \mbox{argmin}_C g $ is non-empty. The compactness of $ C $ is to ensure that $ \mbox{argmin}_C g $ is non-empty and compact which again implies that the problem \ref{sbp4} has a solution.
\end{Remark}

\begin{Remark}
The trickiest part in this scheme is to create such a sequence $\{ \alpha_k\} $. One way to do it may be the following. Construct a sequence $ \{ g(y_k) \} $, with $ y_k \in C $ such that $ g(y_k) \rightarrow \alpha $. By definition $ \alpha \leq g(y_k) $. Thus for any $ \eta_k >0 $, $ k \in \mathbb{N} $ with $ \eta_k \rightarrow 0 $, we have $ \alpha < g(y_k) + \eta_k $. Hence, we can set $ \alpha_k = g(y_k) + \eta_k $. One of key ideas in our algorithm will be to construct the sequence $ \{ y_k \} $. Observe that for the proof of Theorem \ref{gen} we need no compactness assumption on the convex set $ C $. \qed
\end{Remark}

\section{Preliminary Results}
In this section, we state a few results that are useful in developing our algorithm and are needed to understand the convergence analysis. These results are very well known and are given below. The lemma \ref{lemma1} is an application of Theorem 9.8 and Theorem 9.10 given in the book \cite{Beck_book} by Amir Beck.

\begin{Lemma} \label{lemma1}
Let $ g : \mathbb{R}^n \rightarrow \mathbb{R} $ be a smooth convex function, and let $ C \subset \mathbb{R}^n $ be a closed convex set. Also, let $ x \in C $ and $ z = P_C ( x - \beta \nabla g(x) ) $, where $ \beta >0 $. Then
\begin{enumerate}
\item[(i)] $ \langle \nabla g(x), z-x \rangle \leq 0 $.
\item[(ii)] $ \langle \nabla g(x), z-x \rangle = 0 $ if and only if $ x \in \mbox{argmin}_C g $.
\end{enumerate}
\end{Lemma}

\begin{Lemma} \label{lemma4}
Take $\sigma \in (0,1)$, $x \in C$ and $v \in \mathbb{R}^n$ such that $\langle \nabla g(x), v \rangle <0$. Then there exists $\bar{\gamma} <1$ such that
\begin{eqnarray*}
g(x+\gamma v) < g(x) + \sigma \gamma \langle \nabla g(x), v \rangle, \quad \forall \gamma \in (0,\bar{\gamma}]
\end{eqnarray*}
\end{Lemma}

\proof By the definition of differentiabilty of $g$,
\begin{eqnarray*}
\lim\limits_{\gamma \rightarrow 0} \frac{g(x + \gamma v) - g(x)}{\gamma}= \langle \nabla g(x), v \rangle.
\end{eqnarray*}
Now since $\langle \nabla g(x),v \rangle <0$ and $ \sigma \in (0,1)$, we have
\begin{eqnarray*}
\langle g(x), v \rangle < \sigma \langle g(x), v \rangle.
\end{eqnarray*}
Therefore,
\begin{eqnarray*}
\lim\limits_{\gamma \rightarrow 0}\frac{g(x+\gamma v)-g(x)}{\gamma}< \sigma \langle \nabla g(x), v
\rangle.
\end{eqnarray*}
This implies that there exists $\bar{\gamma} <1$ such that
\begin{eqnarray*}
g(x+\gamma v)-g(x)-\gamma \sigma \langle \nabla g(x), v \rangle <0, \quad \forall \gamma \in (0,\bar{\gamma}].
\end{eqnarray*}
Hence,
\begin{eqnarray*}
g(x+\gamma v) < g(x) + \sigma \gamma \langle \nabla g(x), v \rangle, \quad \forall \gamma \in (0,\bar{\gamma}]
\end{eqnarray*} \qed\\
The next lemma is originally from \cite{Part2} named as Lemma 2.1.
\begin{Lemma} \label{rcl}
Let $ f, g : \mathbb{R}^n \rightarrow \mathbb{R} $ are convex functions and $ C $ be a closed, convex subset of $ \mathbb{R}^n $. Also let us assume that the set $ S:= \mbox{argmin}_{S_0} f $ is non-empty and bounded where
$ S_0 := \mbox{argmin}_C g  $. Then for any $ M_1, M_2 \in \mathbb{R} $, the set
\begin{eqnarray*}
K:= C \cap \{ x\in \mathbb{R}^n : g(x) \leq M_1 \} \cap \{ x \in \mathbb{R}^n : f(x) \leq M_2\}
\end{eqnarray*}
is bounded.
\end{Lemma}
\noindent We say that a non-negative sequence $ a_n $ is summable if $ \sum\limits_{n=0}^{\infty} a_n < \infty $ or another way to say that $ \sum\limits_{n=0}^{\infty} a_n $ converges. The following lemma is about that.
\begin{Lemma}\label{fejer}
Let $\{ \alpha_k \}$ be a sequence which is bounded below and $\eta_k \geq 0$ be a summable sequence such that
\[
\alpha_{k+1} \leq \alpha_k + \eta_k
\]
for each $k =0,1,2, \cdots$. Then the sequence $\{ \alpha_k \}$ converges.
\end{Lemma}

\begin{proof}
First note that for $k > l$,
\[
\alpha_k \leq \alpha_l + \sum_{j = l}^{k-1} \eta_j
\]
This immediately implies that the sequence $\{ \alpha_k \}$ is bounded above also. Letting  $k \to \infty$
in the above inequality, we get
\begin{equation}\label{series}
\limsup_{k \to \infty} \alpha_k \leq \alpha_l + \sum_{j=l}^\infty \eta_j
\end{equation}
As $ \{ \eta_k \} $ is a summable seqence, for any $ \varepsilon>0 $ there exists $ N \in \mathbb{N} $ such that
\[
\sum_{j=l}^\infty \eta_j < \varepsilon \quad \forall~ l\geq N.
\]
Then (\ref{series}) together with the above inequality implies that
\begin{equation*}
\limsup_{k \to \infty} \alpha_k < \alpha_l + \varepsilon \quad \forall~ l\geq N.
\end{equation*}
Therefore,
\[
\limsup_{k \to \infty} \alpha_k \leq  \liminf_{l \to \infty} \alpha_l + \varepsilon.
\]
Since this is true for any arbitrary $ \varepsilon $, we can conclude that $ \limsup_{k \to \infty} \alpha_k \leq  \liminf_{l \to \infty} \alpha_l $ and hence the sequence $ \{\alpha_k \} $ converges.
\end{proof}

\section{Algorithm for (SBP) with smooth lower level objective function} \label{smoothalgo}
In this section we are going to present an algorithm for the (SBP) problem, where the lower level function $ g $ is smooth, while there is no restriction on the upper level function $ f $. The key is to use a modified version of the projected gradient method which was used by Iusem \cite{Iusem03} to solve a smooth convex optimization problem. We merge the scheme used by Iusem with the general scheme mentioned in section \ref{gensch} to create an algorithm for solving the (SBP) problem with smooth lower level objective function. Further we shall not require the gradient of the lower level function to be Lipschitz continuous. The best part of this algorithm is that it is easily implementable.
We will now present our algorithm.\\

\noindent {\textbf Algorithm for (SBP):  (SBP-LFS)}\\

\begin{enumerate}
\item $x_0 \in C$, $\eta_k > 0$ such that $\sum \eta_k < \infty$, $\beta_k \in [ \tilde{\beta}, \hat{\beta}]$.

\item Compute $z_k = P_C (x_k  - \beta_k \nabla g(x_k))$.
	\begin{enumerate}
		\item If $\langle \nabla g(x_k), z_k - x_k \rangle = 0$, $\alpha_k = g(x_k) = \min_C g$.
		\item If $\langle \nabla g(x_k), z_k - x_k \rangle < 0$, let
			\[
				l(k) = \min \{ j \in \mathbb{Z}^+ : g(x_k + 2^{-j} (z_k - x_k)) < g(x_k) + \sigma
				2^{-j}  \langle \nabla g(x_k), z_k - x_k \rangle \},
			\]
			where $\sigma \in (0, 1)$. Now compute
			\begin{eqnarray*}
				\gamma_k & =& 2^{- l(k) } \\
				y_k & = &x_k + \gamma_k (z_k - x_k) \\
				\alpha_k &=& g(y_k)
			\end{eqnarray*}
	\end{enumerate}

\item $x_{k+1}$ is chosen to be the $ \epsilon $-minimizer of the following problem
\[
\begin{cases}
\min f(x) \\
g(x) \leq \alpha_k + \eta_k \\
x \in C
\end{cases}
\]

\item Repeat Steps 2 and 3 until convergence.
\end{enumerate}
\begin{Remark} \textrm
We would like to add that we made the above algorithm as (SBP-LFS) to keep in view that the algorithm is for the case where the (SBP) problem has a smooth lower level objective function. This algorithm can be implemented using several stopping criteria. The key is that a natural stopping criteria must involve both the upper level and lower level objectives. Our stopping criteria is based on the fact whether the infimum $ \alpha $ of the lower level problem is easily known or not.\\
If $ \alpha $ is easily known then given $ \varepsilon > 0 $ as a threshold, one can stop the algorithm if
\[
| \alpha_{k+1} - \alpha | \leq \varepsilon \mbox{ and } | f (x_{k+1}) - f(x_k) | \leq \varepsilon \label{SCA} \tag{A}
\]
and choose $ x_{k+1} $ as the approximate solution. We will call this as Stopping Criteria (A).\\
If $ \alpha $ can not be known easily, we can devise a stopping criteria as follows. Since we construct $ \alpha_k $ such that $ \alpha_k \rightarrow \alpha $, the sequence $ \{ \alpha_k \} $ is a Cauchy sequence. Thus from the practical point of view for a given threshold $ \varepsilon > 0 $, we can stop the algorithm if
\[
| \alpha_{k+1} - \alpha_k | \leq \varepsilon \mbox{ and } | f (x_{k+1}) - f(x_k) | \leq \varepsilon \label{SCB} \tag{B}
\]
and as before take $ x_{k+1} $ as our approximate solution. We will mention this condition as Stopping Criteria (B).
\end{Remark}

\begin{Theorem} \label{smcon}
Assume that the solution set $S$ of problem \eqref{sbp4} is non-empty.
Let $\{ x_k \}$ be a sequence generated by the algorithm (SBP-LFS).  Let $\bar{x}$ be a cluster
point of the sequence $\{ x_k \}$. Then $\bar{x}$ is a solution of \eqref{sbp4}.
%
%
\end{Theorem}

\begin{proof}
Note that
\[
g(x_k) \leq \alpha_{k-1} + \eta_{k-1}.
\]
If $x_{k}$ is chosen according to Step 2(a), then $\alpha_{k} = g(x_{k})$, in which case
\[
\alpha_{k} \leq \alpha_{k-1} + \eta_{k-1}.
\]
Suppose $x_{k}$ is chosen according to Step 2(b). Then
\[
\alpha_k = g(y_k) \leq g(x_k)
\]
which immediately implies
\[
\alpha_{k} \leq \alpha_{k-1} + \eta_{k-1}.
\]
Note also that for any $k$, $\alpha_k$ is either $g(x_k)$ or $g(y_k)$. Also note that
$g(y_k) \leq g(x_k)$.

From Lemma \ref{fejer}, we have that $\alpha_k \to \bar{\alpha}$ for some $\bar{\alpha}$.

Since $\bar{x}$ is a cluster point, there is a subsequence of $\{x_k\}$ which converges to $\bar{x}$.
Without loss of generality, let us denote this subsequence by $\{x_k \}$ itself.  By the definition,
we see that $z_k \to \bar{z}$ along this subsequence where $\bar{z} = P_C (\bar{x} - \beta \nabla g(\bar{x}))$ and $\beta$ is limit of $\beta_k$ as $k\to \infty$.
Now two cases arise:
one in which there is a subsequence along which $ \langle \nabla g(x_k), z_k - x_k \rangle = 0$
and the second one where $\langle \nabla g(x_k), z_k - x_k \rangle < 0$ for all $ k $ after a finitely many $ k $'s which we can ignore.

In Case 1, we have $\langle \nabla g(\bar{x}), \bar{z} - \bar{x} \rangle = 0$ and hence $\bar{x}$
minimizes $g(x)$ over $C$. Moreover,
\[
g(\bar{x}) = \lim_{k \to \infty} g(x_k) \leq \bar{\alpha}.
\]
and hence $\bar{x}$ solves \eqref{sbp4}.

Therefore we consider the Case 2. Therefore we can assume that there is no subsequence along which
$\langle \nabla g(\bar{x}), \bar{z} - \bar{x} \rangle = 0$ and hence
$\langle \nabla g(x_k), z_k - x_k \rangle < 0$ for all $k$. Thus
$\langle \nabla g(\bar{x}), \bar{z} - \bar{x} \rangle \leq 0$. It is also obvious
that
\[
g(\bar{x}) = \lim_{k \to \infty} g(x_k) \leq \bar{\alpha}
\]
It remains to show that  $\bar{\alpha} = \min_{x \in C} g(x)$.
By the definition of $\gamma_k$, we have
\[
g(y_k) < g(x_k) + \sigma \gamma_k  \langle \nabla g(x_k), z_k - x_k \rangle
\]
Therefore,
\[
0 < - \sigma \gamma_k  \langle \nabla g(x_k), z_k - x_k \rangle  < g(x_k) - g(y_k)
\leq \alpha_{k-1} + \eta_{k-1} - \alpha_k .
\]
Since the RHS tends to zero as $k \to \infty$, we must have
\[
\bar{\gamma} \langle \nabla g(\bar{x}), \bar{z} - \bar{x} \rangle = 0.
\]
If $\bar{\gamma} \neq 0$, then $\langle \nabla g(\bar{x}), \bar{z} - \bar{x} \rangle = 0$ and hence
$\bar{x}$ minimizes $g(x)$ over $C$.

Therefore, we consider the case in which $\bar{\gamma} = 0$.
Since $\bar{\gamma} = 0$, $l(k) \to \infty$ as $k \to \infty$. Thus for each positive integer $M > 0$,
there exists $k_0$ such that for all $k \geq k_0$, we have
\begin{equation}\label{Mchoice}
g(x_k + 2^{-M} (z_k - x_k))  \geq g(x_k) + \sigma 2^{-M} \langle \nabla g(x_k), z_k - x_k \rangle.
\end{equation}
Letting $k \to \infty$, we have for each positive integer $M$
\[
g(\bar{x} + 2^{-M} (\bar{z} - \bar{x})) \geq g(\bar{x}) + \sigma 2^{-M} \langle \nabla g(\bar{x}, \bar{z} - \bar{x} \rangle.
\]
Now suppose
$\langle \nabla g(\bar{x}), \bar{z} - \bar{x} \rangle < 0$.
By Lemma \ref{lemma4}, there eixsts $\hat{\gamma}  < 1$ such that
\[
g(\bar{x} + \gamma (\bar{z} - \bar{x})) < g(\bar{x}) + \sigma \gamma \langle \nabla g(\bar{x}, \bar{z} - \bar{x} \rangle
\]
for all $\gamma < \hat{\gamma}$.
This contradicts \eqref{Mchoice} for all large $M$ such $2^{-M} < \hat{\gamma}$ and hence
$\langle \nabla g(\bar{x}), \bar{z} - \bar{x} \rangle = 0$, proving the fact that $\bar{x}$ minimizes $g(x)$ over $C$.\\
Let us now consider $ x^{*} \in \mbox{arg}\min\limits_C g $, then $ x^{*} $ is a feasible point for the problem $(P_k)$. As we have assumed $ x_{k+1}$ to be an $ \epsilon-$minimizer of $(P_k)$, we can conclude that
\[
f(x_{k+1}) \leq f(x_{k+1}^{*}) + \epsilon \leq f(x^{*}) + \epsilon
\]
where $ x_{k+1}^{*} $ is a solution of $(P_k)$. Now as $ k \rightarrow \infty $, from the above inequality we get
\[
f(\bar{x}) \leq f(x^{*}) + \epsilon,
\]
which implies that $ \bar{x} $ is an $ \epsilon$- approximate solution of the (SBP) problem. Further note that if we consider a decreasing sequence $ \epsilon_k \downarrow 0 $ and use $ \epsilon_k $ instead of $ \epsilon $ in the problem $(P_k)$, then we will get $ f(\bar{x}) \leq f(x^{*}) $. Which indicates that the algorithm (SBP-LFS) actually leads to a solution of the (SBP) problem.
\end{proof}

In the above result, we assumed the existence of a cluster point for the iterations. This is an obvious consequence of the assumption that either the convex set $C$ or the sublevel sets of either of the objective functions is compact. We can also derive this result when the solution set of \eqref{sbp4} is assumed to be compact. We state this as the following theorem.

\begin{Theorem}
Let the solution set of \eqref{sbp4} is non-empty and bounded. Then the iterates of our algorithm are bounded and hence has a subsequence which converges to a minimizer of \eqref{sbp4}.
\end{Theorem}

\begin{proof}
From the proof of Theorem \ref{smcon}, we know that the sequence $\{ \alpha_k \}$ is bounded and
hence the sequence of iterates $\{ x_k \}$ belong to the set
\[
K:= C \cap \{ x : f(x) \leq M_1 \} \cap \{ g(x) \leq M_2 \}
\]
for some positive numbers $M_1$ and $M_2$. Now by lemma \ref{rcl} we can say that if the solution set of the (SBP) problem is bounded then $ K $ is also bounded. This implies that $ \{ x_k \} $ is a bounded sequence and hence has a convergent subsequence. Then theorem \ref{smcon} ensures that the limit of this convergent subsequence belongs to the solution set of the (SBP) problem and hence proving the theorem.
\end{proof}

\section{Numerical Experiments}
The numerical experiments presented here were coded in MATLAB version R2006a and run on a computer with an Intel core i7 vpro processor running at 3.10 GHz using 8 GB of RAM, running WINDOWS version 10. To solve the sub-problem in our algorithm we have used the MATLAB software CVX designed to solve convex optimization problem. Also the data presented in this section are rounded up to 6 decimal numbers. To the best of our knowledge we are not aware of any library of test functions for (SBP) problems. A primary reason for this could be that literature on the algorithms for (SBP) problem has just began to grow and thus most authors devise their own problems or consider a problem from applications, see for example \cite{Beck}, \cite{Helou}.
\subsection{(SBP) with affine upper level objective} \label{example1}
Here we consider the following simple bilevel optimization problem with multiple lower level solutions.
\begin{equation}\label{ex1}
\begin{cases}
\min x_1 + x_2 -1  \\
\mbox{subject to } \\
(x_1, x_2) \in \mbox{argmin} \{ g(y_1, y_2) : (y_1, y_2 )\in C \},
\end{cases}
\end{equation}
Where $ C = \{ (y_1, y_2) \in \mathbb{R}^2 : y_1^2 + y_2^2 \leq 2, -3 \leq y_1, y_2 \leq 0.5 \}$ and
\begin{eqnarray*}
 g( y_1, y_2 ) =
\begin{cases}
( \sqrt{y_1^2 + y_2^2 } - 1 )^2 \quad &\mbox{if} \quad \sqrt{y_1^2 + y_2^2 } \geq 1 \\
0 \quad &\mbox{if} \quad \sqrt{y_1^2 + y_2^2 } < 1.
\end{cases}
\end{eqnarray*}
Then clearly, $ \mbox{argmin}_C g  = \{( y_1, y_2 ) \in \mathbb{R}^2 : y_1^2 + y_2^2 \leq 1, -3 \leq y_1, y_2 \leq 0.5 \} $. A simple calculation will show that $ ( - \frac{1}{ \sqrt{2} }, - \frac{1}{ \sqrt{2} } )$ is the solution of this simple bilevel problem.\\
Note that in this example both the lower and upper level functions are differentiable convex functions with gradient Lipschitz continuous and the constraint set is also convex. We have applied the algorithm for the smooth (SBP) problem presented in section \ref{smoothalgo} and the algorithm presented by Solodov in his paper \cite{Solodov1} to compare the obtained results and also the time and iteration number taken by the algorithms to attain the desired accuracy.\\

First we provide a small proof of the fact that the objective functions in this problem satisfy Solodov's assumption in \cite{Solodov1} i.e. both the functions have Lipschitz continuous gradient. Note that the upper level function is an affine function hence has constant gradient which automatically satisfies the criteria. So let us focus on the lower level function $ g(y_1, y_2)$ and see how the gradient of this function given by
\begin{eqnarray*}
 \nabla g( y_1, y_2 ) =
\begin{cases}
	2(y_1,y_2) - 2 \frac{(y_1, y_2)}{\sqrt{y_1^2 + y_2^2 }} \quad &\mbox{if} \quad \sqrt{y_1^2 + y_2^2 } \geq 1 \\
	0 \quad &\mbox{if} \quad \sqrt{y_1^2 + y_2^2 } < 1.
\end{cases}
\end{eqnarray*}
is a Lipschitz continuous function. Let $ D $ be the unit disc in $ \mathbb{R}^2 $, then
\begin{equation*}
\nabla g (y_1,y_2) = 2 [(y_1, y_2) - P_D(y_1, y_2)].
\end{equation*}
Now we know that $ P_D $, the projection map on a convex set is non-expansive and therefore Lipschitz continuous also. Which implies that $ \nabla g $ is also a Lipschitz continuous function. We can also see this clearly from the following inequalities.
\begin{eqnarray*}
\| \nabla g (x_1,x_2) - \nabla g (y_1,y_2) \| &=& 2 \| (x_1, x_2) - P_D(x_1, x_2) - (y_1, y_2) + P_D(y_1, y_2) \| \\
&\leq& 2 \|(x_1, x_2) - (y_1, y_2)\| + 2\| P_D(x_1, x_2)- P_D(y_1, y_2)\|\\
&\leq& 4\|(x_1, x_2) - (y_1, y_2)\|.
\end{eqnarray*}
\noindent Next we present the results of the algorithm (SBP-LFS) and Solodov's algorithm [see \cite{Solodov1}] when applied to the problem (\ref{ex1}). Note that there is no such parameter like $ \eta_k $ in Solodov's algorithm so the results of Solodov's algorithm are given in terms of the threshold ($ \epsilon $) value.\\
\noindent Though $ \inf_C g = \alpha = 0 $ in this case, it may not be known for an (SBP) problem of this form. So we will proceed with the Stopping Criteria (B) and the result is given as follows.\\

\noindent \begin{tabular}{|l|l|l|l|l|l|l|}\hline

\thead{Threshold\\ ($\varepsilon$)}	& $ \eta_k $	& Algorithm	& \thead{Solution\\($x^*$)}	& \thead{Optimal value\\ $ f(x^*) $}	& \thead{Iteration\\ number}	& \thead{CPU Time\\(in Sec.) }\\
\hline
$ 10^{-5} $	& $ \frac{1}{10^k}$	& SBP-LFS	& $(-0.7071,   -0.7071)$	& $ -2.414214 $	& $18$	& 40.591314 \\
\hline
$ 10^{-5} $	& $ \frac{1}{100^k}$	& SBP-LFS	& $(-0.7071,   -0.7071)$	& $ -2.414214 $	& $ 17 $	& 25.205724 \\
\hline
$ 10^{-5} $	& $ \frac{1}{1000^k}$	& SBP-LFS	& $(-0.7071,   -0.7071)$	& $ -2.414214 $	& $ 17 $	& 25.647749 \\
\hline
$ 10^{-5} $	& -	& Solodov	& $(-0.7087, -0.7087)$	& $ -2.417368 $	& 318	& 504.377589 \\
\hline
$ 10^{-7} $	& $ \frac{1}{10^k}$	& SBP-LFS	& $(-0.7071,   -0.7071)$	& $ -2.414214 $	& $ 25 $	& 37.663035 \\
\hline	
$ 10^{-7} $	& $ \frac{1}{100^k}$	& SBP-LFS	& $(-0.7071,   -0.7071)$	& $ -2.414214 $	& $ 24 $	& 37.805480 \\
\hline
$ 10^{-7} $	& $ \frac{1}{1000^k}$	& SBP-LFS	& $(-0.7071,   -0.7071)$	& $ -2.414214 $	& $ 24 $	& 38.073170 \\
\hline
$10^{-7}$	& -	& Solodov	& $ (-0.7073, -0.7073) $	& $ -2.414530 $	& $ 3164 $	& $ 5760.272470 $\\
\hline
$ 10^{-9} $	& $ \frac{1}{10^k}$	& SBP-LFS	& $(-0.7071,   -0.7071)$	& $ -2.414214 $	& $30 $	& 47.205540 \\
\hline
$ 10^{-9} $	& $ \frac{1}{100^k}$	& SBP-LFS	& $(-0.7071,   -0.7071)$	& $ -2.414214 $	& $ 29 $	& 44.255782 \\
\hline
$ 10^{-9} $	& $ \frac{1}{1000^k}$	& SBP-LFS	& $(-0.7071,   -0.7071)$	& $ -2.414214 $	& $ 28 $	& 43.404657 \\
\hline
$ 10^{-9} $	& $ -	$	& Solodov	& $(-0.7071,   -0.7071)$	& $ -2.414214 $	& $ \approx 10000 $	& - \\
\hline

\end{tabular}
$~$\\

\begin{Remark}
 While executing the algorithm, we have noted that the solution does not depend on the choice of the initial point selection for the algorithm. Only there can be slight variation in the execution time for different initial points. We have also obtained $ \min_C g = 0 $ for all the cases mentioned in the above table.
\end{Remark}

\subsection{Least $ l_1 $-norm solution of system of linear equations}
Finding sparse solution of linear system like $ Ax = b $ has been a major
requirement in many modern applications and notably so in compressed sensing. The sparse solutions are often obtained by minimizing the $ l_1$-norm subject to the constraints $ Ax = b $. For a nice survey of application of this problem to compressed sensing see for example Bryan and Leise \cite{bl-2013}. It might however happen that it might not be possible to have an exact solution of $ Ax = b $ or it might not have a solution as in the case of over-determined systems. Then we need to be satisfied with least square solutions. In fact if we seek a sparse least square solution then we are essentially solving a simple bilevel problem with a non-smooth upper level objective and a smooth lower level objective. It is also well-known that the least square approach in itself is an important approach to find a solution to the system $ Ax = b $ and thus if we seek a sparse solution we are again in the simple programming set up. The following numerical experiments show that simple bilevel programming approach indeed appears to be an effective way to compute a sparse solution of a linear system $ Ax = b$. In the first example we find the least norm solution of a system of equations with $ 4 $ variables and $ 3 $ equations. The second example is about solving the same problem with $ 15 $ variables and $ 5 $ equations. In both the examples $ Ax = b $ has exact solutions. In the third example $ Ax = b $ does not have any solution, so we aim to solve $ \min \|x\|_1 \mbox{ subject to } x \in \mbox{argmin}\{\|Ay - b\|^2\}$.
\begin{enumerate}[(i)]
	\item Consider the problem of finding least $l_1$-norm solution of the linear system $Ax = b$,
where
\[ A=
\begin{bmatrix}
\phantom{-}1	& \phantom{-}2	& -3	&\phantom{-} 1 \\
\phantom{-}3 	&           -1	& -2	&           -4 \\
\phantom{-}2	& \phantom{-}3	& -5	& \phantom{-} 1	
\end{bmatrix}
\quad \mbox{and} \quad
b=
\begin{bmatrix}
          -2 \\
\phantom{-}1 \\
          -3
\end{bmatrix}.
\]
This problem can be written as a simple bilevel problem as follows
\begin{equation}\label{ex2}
\begin{cases}
\min \| x \|_1  \\
\mbox{s.t. } x \in \mbox{argmin} \{ \| Ax -b \|^2 : x\in \mathbb{R}^4 \}.
\end{cases}
\end{equation}
The solution of this system of equations is given by $ x_1 = s+ t $, $ x_2 = s -t -1 $, $ x_3 = s $, $ x_4 = t $; where $ s, t \in \mathbb{R} $ and $ x= (x_1, x_2, x_3, x_4)$.\\
Even though the upper-level function of this (SBP) problem is non-smooth we can still apply our algorithm to this problem as the only requirement for the algorithm is the smoothness of the lower-level function. Note that in this case the minimum value of the lower level problem is known to us and $ \min \{ \| Ax -b \|^2 : x\in \mathbb{R}^4 \} = 0 $, considering the fact that the system of equations mentioned above has at least one solution. By the convergence analysis of the algorithm we also know that $ f(x_k) = \| x_k \|_1 $ converges. For any given $ \varepsilon > 0 $, we have used these information to implement the stopping criteria \eqref{SCA}. The results for different thresholds and different choices of $ \eta_k $ are presented in the following table.

\begin{flushleft}
\begin{tabular}{|c|c|c|c|c|c|c|}
\hline

\thead{Threshold\\ ($\varepsilon$)}	& $ \eta_k $	& \thead{Solution\\($x^*$)}	& $ \| Ax^* -b \|^2 $	& \thead{Optimal value\\ $ \| x^* \|_1 $}	& \thead{Iteration\\ number} & \thead{CPU Time\\(in sec.)}\\
\hline
$ 10^{-5} $	& $ \frac{1}{10^k} $	& $(0, -1, 0, 0)$	& $ 3.441514e-10 $	& $ 9.999846e-01$	& $ 12 $    & $6.697958$ \\
\hline
$ 10^{-5} $	& $ \frac{1}{100^k} $	& $(0, -1, 0, 0)$	& $ 1.078392e-10 $	& $ 9.999849e-01 $	& $ 12 $  & $ 6.466460$ \\	
\hline
$ 10^{-5} $	& $ \frac{1}{1000^k} $	& $(0, -1, 0, 0)$	& $ 1.508332e-10 $	& $ 9.999848e-01 $	& $ 12 $ & $ 6.257101$\\	
\hline
$ 10^{-7} $	& $ \frac{1}{10^k} $	& $(0, -1, 0, 0)$	& $ 3.828847e-10 $	& $ 9.999846e-01 $	& $ 13 $   & $ 6.934747$\\
\hline
$ 10^{-7} $	& $ \frac{1}{100^k} $	& $(0, -1, 0, 0)$	& $ 3.826991e-10 $	& $ 9.999845e-01 $	& $ 14 $  & $ 7.490094$\\
\hline
$ 10^{-7} $	& $ \frac{1}{1000^k} $	& $(0, -1, 0, 0)$	& $ 3.832178e-10 $	& $ 9.999846e-01 $	& $ 14 $ & $ 7.440750$ \\
\hline
$ 10^{-9} $	& $ \frac{1}{10^k} $	& $(0, -1, 0, 0)$	& $ 3.848853e-10 $	& $ 9.999846e-01 $	& $ 35 $   & $ 20.496428$\\
\hline
$ 10^{-9} $	& $ \frac{1}{100^k} $	& $(0, -1, 0, 0)$	& $ 3.843531e-10 $	& $ 9.999846e-01 $	& $ 53 $  & $ 34.054234$\\
\hline
$ 10^{-9} $	& $ \frac{1}{1000^k} $	& $(0, -1, 0, 0)$	& $ 3.855139e-10 $	& $ 9.999845e-01 $	& $ 28 $ & $ 15.718556$\\
\hline

\end{tabular}
\end{flushleft}

\item Another problem which establishes our result is given by
\begin{equation}\label{ex2.1}
\begin{cases}
\min \| x \|_1  \\
\mbox{s.t. } x \in \mbox{argmin} \{ \| Ax -b \|^2 : x\in \mathbb{R}^{15} \},
\end{cases}
\end{equation}
where
\[
A = \begin{bmatrix}
\phantom{-}3 & \phantom{-}1 & \phantom{-}2 & -1 & \phantom{-}1 & \phantom{-}1 & -2 & \phantom{-}1 & -2 & -3 & \phantom{-}1 & -1 & \phantom{-}2 & -1 & -1\\
\phantom{-}1 & -2 & -1 & \phantom{-}1 & \phantom{-}2 & -1 & -1 & -2 & -3 & \phantom{-}2 & \phantom{-}1 & -5 & -1 & \phantom{-}1 & -2\\
-2 & -1 & \phantom{-}1 & -1 & \phantom{-}1 & -1 & \phantom{-}2 & \phantom{-}1 & -3 & \phantom{-}1 & \phantom{-}2 & \phantom{-}2 & \phantom{-}3 & \phantom{-}2 & \phantom{-}2 \\
-3 & -1 & \phantom{-}1 & \phantom{-}2 & -5 & -6 & \phantom{-}7 & -1 & -2 & -3 & \phantom{-}1 & -2 & \phantom{-}3 & \phantom{-}1 & \phantom{-}3\\
\phantom{-}4 & -1 & \phantom{-}2 & \phantom{-}4 & \phantom{-}5 & -1 & -2 & \phantom{-}1 & -3 & \phantom{-}1 & -1 & -2 & -3 & \phantom{-}4 & \phantom{-}5
\end{bmatrix}
\]
and $ b = \begin{bmatrix}
1\\
0\\
3\\
2\\
12
\end{bmatrix}.\\
$
Using the same stopping criteria as used in the above example we get the following results for $ \varepsilon = 10^{-5} $ and different choices for $ \eta_k $.

\begin{flushleft}
\begin{tabular}{|c|c|c|c|c|c|c|c|}
\hline

\thead{Threshold\\ ($\varepsilon$)} & $ \eta_k $	& $ \| Ax^* -b \|^2 $	& \thead{Optimal value\\ $ \| x^* \|_1 $}	& \thead{Iteration number} & \thead{CPU Time\\(in sec)}\\
\hline
$ 10^{-5} $ & $ \frac{1}{10^k} $	& $ 3.152108e-10 $	& $ 2.576560 $	& $ 17 $ & 8.542113 \\
\hline
$ 10^{-5} $ & $ \frac{1}{100^k} $	& $ 3.135185e-10 $	& $ 2.576563 $	& $ 17 $ & 8.662664 \\
\hline
$ 10^{-5} $ & $ \frac{1}{1000^k} $	& $ 3.147713e-10 $	& $ 2.576563 $	& $ 17 $ & 8.527413 \\
\hline
$ 10^{-7} $ & $ \frac{1}{10^k} $	& $ 9.196033e-11 $	& $ 2.576562 $	& $ 22 $ & 11.747009 \\
\hline
$ 10^{-7} $ & $ \frac{1}{100^k} $	& $ 5.471720e-11 $	& $ 2.576563 $	& $ 18 $ & 9.102962 \\
\hline
$ 10^{-7} $ & $ \frac{1}{1000^k} $	& $ 5.870428e-10 $	& $ 2.576563 $	& $ 23 $ & 12.486632 \\
\hline
$ 10^{-9} $ & $ \frac{1}{10^k} $	& $ 5.847742e-11 $	& $ 2.576563 $	& $ 107 $ & 61.940890 \\
$ 10^{-9} $ & $ \frac{1}{100^k} $	& $ 5.837643e-11 $	& $ 2.576563 $	& $ 62 $ & 34.612342 \\
\hline
$ 10^{-9} $ & $ \frac{1}{1000^k} $	& $ 5.833193e-11 $	& $ 2.576563 $	& $ 38 $ & 20.896115 \\
\hline

\end{tabular}
\end{flushleft}
with $ x^* = (0.3857, 0, 0, 0.0506, 0.6820,0,0,0,0.2233,0,0,0,0,0,1.2350) $.\\
We tried to improve the results by taking $ \varepsilon = 10^{-7} $, but the process was more time consuming without improving the outputs $(\| Ax^* -b \|^2 \mbox{ and } \| x^* \|_1)$ much. So we stopped the algorithm after $ 1001 $ and $ 10001 $ steps. In both the cases we got the results which is not very different from what we got for $ \varepsilon = 10^{-5} $ in $ 16 $ steps. To be precise we got $ x^* $ and $ \| x^* \|_1 $ same as we got in $ 16 $ steps with slightly improved values of $ \| Ax^* -b \|^2 $ as $ 6.265654e-7 $ and $ 6.122150e-07 $ in $ 1001 $ and $ 10001 $ steps respectively. So we can conclude that the best result we got for this problem is the one we got for $ \varepsilon = 10^{-5} $ in $ 16 $ steps which is not a bad one considering that we managed to obtain $ x^* $ with $ 5 $ non-zero entries and $ 10 $ zero entries.

\item We tried our algorithm on another least square problem where $ Ax = b $ does not have an exact solution. Here we show that our algorithm successfully finds a minimizer of the problem $ \min \|Ax-b\|^2 $ with least norm. We consider the least square problem with
\[
A = \begin{bmatrix}
\phantom{-}2	& \phantom{-}1	& -1\\
\phantom{-}0	& \phantom{-}2	& \phantom{-}1\\
-1	& \phantom{-}2	& \phantom{-}0\\
-2	& -1	& \phantom{-}2\\
\phantom{-}1	& \phantom{-}3	& -3\\
\phantom{-}3	& -1	& \phantom{-}0\\
-3	& \phantom{-}1	& -2\\
\phantom{-}1	& -4	& \phantom{-}1\\
\phantom{-}5	& \phantom{-}2	& -1\\
\phantom{-}4	& \phantom{-}0	& -3
\end{bmatrix}
\mbox{ and } b = \begin{bmatrix}
\phantom{-}0\\
\phantom{-}2\\
\phantom{-}3\\
\phantom{-}1\\
\phantom{-}0\\
\phantom{-}2\\
\phantom{-}4\\
-2\\
\phantom{-}5\\
\phantom{-}2
\end{bmatrix}.
\]
The best result by applying our algorithm is obtained in $ 12 $ iterations with the approximated solution $ x^* = (0.2334, 0.6718, -0.0153) $, $ \|x^*\|_1 = 0.9205411 $ and the lower level functional value $ \|Ax^* - b\|^2 = 42.57734 $. Following are the details.

\begin{flushleft}
\begin{tabular}{|c|c|c|c|c|c|}
\hline

$\varepsilon$	& $ \eta_k $	& $ \| Ax^* -b \|^2 $	& \thead{Optimal value\\ $ \| x^* \|_1 $}	& \thead{Iteration\\ number} & \thead{CPU Time\\(in sec.)}\\
\hline
$ 10^{-5} $	& $ \frac{1}{10^k} $	& $ 4.257734e+01 $	& $ 9.205411e-01$	& $ 12 $    & $11.326960$ \\
\hline
$ 10^{-5} $	& $ \frac{1}{100^k} $	& $ 4.257734e+01 $	& $ 9.205412e-01 $	& $ 12 $  & $ 6.133314$ \\	
\hline
$ 10^{-5} $	& $ \frac{1}{1000^k} $	& $ 4.257734e+01 $	& $ 9.205412e-01 $	& $ 12 $ & $ 6.049399$\\	
\hline

\end{tabular}
\end{flushleft}

\end{enumerate}

\subsection{Estimating distance of a given point from the solution set of a optimization problem } \label{example3}
A very important issue in convex optimization theory is to estimate the distance of a feasible point to the solution set of the problem. There is a very broad theory of error bounds for finding the upper estimate of the distance of a point $x \in \mathbb{R}^n$ from the set $C$, given by given convex inequalities i.e.
\begin{eqnarray}\label{1.17}
C= \{x \in \mathbb{R}^n: g_i (x) \leq 0, i=1,2,\dots,m\}
\end{eqnarray}
where each $g_i: \mathbb{R}^n \rightarrow \mathbb{R}$ $(i=1,2,\dots,m)$, is a convex function. It has been established for example in \cite{Manga98} that if the Slater constraint qualification holds along with some additional asymptotic conditions allow us to show that there exists a constant $\gamma>0$ such that
\begin{eqnarray*}
d(x,C) \leq \gamma\|f(x)_+\|
\end{eqnarray*}
where $f(x)_+= (\max\{g_1(x),0\}, \dots, \max\{g_m(x),0\})$. Now if we consider the problem of minimizing the function $\phi$ over the set as given by (\ref{1.17}), then in order to estimate the distance of a point $x$ from the argmin set it is tempting to write the argmin set as
\begin{eqnarray}
\mbox{argmin}_C \phi = \{x: \phi(x)- \alpha, g_1(x),g_2(x), \dots, g_m(x)\}
\end{eqnarray}
where $\alpha= \inf _C \phi$. Then one might be again tempted to use the error bound that we previously discussed and apply it in the case. However, the Slater condition fails if $ \alpha = \min_C \phi $. In fact if $\phi$ is strongly convex then an error bound can be obtained by using the machinery of gap functions.\\
However if $\phi$ is just convex and not strongly convex, then it is indeed a difficult problem to estimate the distance of a point $x$ from $\mbox{argmin}_C \phi$. The simple bilevel programming model approx to be a nice approach to solve this issue. Though we may not be able to generate an upper estimate of the distance by using the bilevel approach but it does provide an estimate of the distance to a desired level of precision.\\

Given a point $x_0$, calculating the distance of $x_0$ from the solution set of the following optimization problem
\begin{center}
$\min g(x)$\\
$x \in C$
\end{center}
where $g: \mathbb{R}^n \rightarrow \mathbb{R}$ is a differentiable convex function and $C \subset \mathbb{R}^n$ is a convex, compact set is same as solving the following bilevel problem
\begin{center}
$\min \frac{1}{2}\|x-x_0\|^2$\\
$s.t. x \in \mbox{argmin}_C g$
\end{center}
Now if we take $f(x)= \frac{1}{2}\|x-x_0\|^2$ and generate iterative points $x_k$ by the algorithm mentioned in this chapter, then $f(x_k)$ converges to the optimum value $\frac{1}{2}\|\bar{x} -x_0\|^2$ which gives the required distance of $x_0$ from the solution set.

As we have discussed above, calculating the distance of a given point from the solution set of an optimization problem can be represented as a simple bilevel programming problem. Here we will apply our algorithm for the smooth (SBP) problem to estimate the distance of a given point say $ a $ from the solution set of the following problem.
\begin{eqnarray} \label{ex3}
\min f(x) \nonumber \\
x \in C,
\end{eqnarray}
where $ C = \{ x= (x_1, x_2) \in \mathbb{R}^2 : x_1 + x_2 \leq 5 , x_1^2 + x_2^2 \leq 25 \} $ and $ f(x) = \max \{ x_1 -2 , 0 , x_2 -2 \} $. Clearly being the maximum of linear functions, $ f $ is also a convex function. We have used the stopping criteria \eqref{SCA} as it is also used in example (\ref{example1}). Let $ S $ be the solution set of the problem (\ref{ex3}). Then the following table represents the distance of a given point $ a $ from $ S $ \textit{i.e.} $ \mbox{dist}(a,S) $ and the point $ x_a \in S $ such that $ \mbox{dist}(a,S) = \| x_a - a \| $.
\bigskip\par

\begin{tabular}{|c|c|c|c|c|c|c|}
\hline
\thead{Point \\ ($ a $)}	& \thead{Threshold\\ ($\varepsilon$)}	& $ \eta_k $	& \thead{Solution point\\($ x_a $)}	& \thead{$ \mbox{Dist}(a,S) $\\ $ \| x_a - a \| $}	& \thead{Iteration\\No.} & \thead{CPU Time\\(in sec)}\\
\hline

$ (0, -3)$	& $ 10^{-5} $	& $ \frac{1}{10^k}$	& $ (0, -3) $	& $ 9.363386e-12 $	& $ 2 $   & 2.388520 \\
\hline
$ (0, -3)$	& $ 10^{-5} $	& $ \frac{1}{100^k}$	& $ (0, -3) $	& $ 2.161774e-11 $	& $ 2 $  & 2.336804\\
\hline

$ (0, -3)$	& $ 10^{-7} $	& $ \frac{1}{10^k}$	& $ (0, -3) $	& $ 9.363386e-12 $	& $ 2 $   & 2.316111 \\
\hline
$ (0, -3)$	& $ 10^{-7} $	& $ \frac{1}{100^k}$	& $ (0, -3) $	& $ 2.161774e-11 $	& $ 2 $  & 2.307831 \\
\hline
$ (2, 3)$	& $ 10^{-5} $	& $ \frac{1}{10^k}$	& $ (1.9999, 2) $	& $ 9.999990e-01 $	& $ 7 $    & 13.399663 \\
\hline
$ (2, 3)$	& $ 10^{-5} $	& $ \frac{1}{100^k}$	& $ (1.9999, 2) $	& $ 1 $	& $ 5 $  & 5.867909 \\
\hline
$ (2, 3)$	& $ 10^{-7} $	& $ \frac{1}{10^k}$	& $ (1.9999, 2) $	& $ 1 $	& $ 9 $   & 10.718645 \\
\hline
$ (2, 3)$	& $ 10^{-7} $	& $ \frac{1}{100^k}$	& $ (1.9999, 2) $	& $ 1 $	& $  6 $ & 7.061641 \\
\hline
$ (0, 3)$	& $ 10^{-5} $	& $ \frac{1}{10^k}$	& $ (0, 2) $	& $ 9.999990e-01 $	& $  7 $    & 8.285713\\
\hline
$ (0, 3)$	& $ 10^{-5} $	& $ \frac{1}{100^k}$	& $ (0, 2) $	& $ 1 $	& $ 5 $   & 5.673931 \\
\hline
$ (0, 3)$	& $ 10^{-7} $	& $ \frac{1}{10^k}$	& $ (0, 2) $	& $ 1 $	& $ 9 $    & 10.224498 \\
\hline
$ (0, 3)$	& $ 10^{-7} $	& $ \frac{1}{100^k}$	& $ (0, 2) $	& $ 1 $	& $ 6 $   & 6.891023 \\
\hline
\end{tabular}

\subsection{Markowitz Portfolio optimization model}
In \cite{Beck}, Beck and Sabach considered the Markowitz Portfolio optimization problem \cite{Portfolio} for their numerical experimentation. We tried to adapt the problem and fit it into the Simple bilevel framework. The specific example used in \cite{Beck}, can be remodelled as
\begin{eqnarray}
    &&\min \| x-a\|^2 \nonumber\\
    &&\mbox{subject to } x \in \mbox{argmin} \{ \langle x, Ax \rangle : \sum x_i = 1, \langle \mu,x \rangle  \geq r_0, x_i\geq 0 \},
\end{eqnarray}\label{MP}
where $\mu, x \in \mathbb{R}^8 $, $ r_0 \in \mathbb{R} $ and $ A \in \mathbb{R}^{8\times 8}$. We wanted to use the same data as mentioned in \cite{Beck}, but failed to do so as the matrix $A $ mentioned in their numerical example has negative eigen value which makes the lower level problem non-convex. Beck and Sabach claimed that the matrix $ A $ in their example does not have any negative eigen value but it has negative eigen value with small absolute value which may be the reason behind considering those negative eigen values as zero which justifies their claim. As we have used CVX in our program to implement the algorithm we could not neglect these negative eigen values as CVX identifies it. So we traced back to the data at http://www.princeton.edu/~rvdb/ampl/nlmodels/markowitz/ which was used in \cite{Beck}  and created the positive semi-definite covariance matrix which is given by
\[
A = \begin{bmatrix}
    \phantom{-}0.0009   &-0.0001    &\phantom{-}0.0001    &\phantom{-}0.0001   &-0.0003   &\phantom{-}0.0003   &-0.0013    &\phantom{-}0.0008\\
   -0.0001   &\phantom{-}0.0232    &\phantom{-}0.0113    &\phantom{-}0.0106    &\phantom{-}0.0118    &\phantom{-}0.0115    &\phantom{-}0.0110   &-0.0141\\
    \phantom{-}0.0001    &\phantom{-}0.0113    &\phantom{-}0.0283    &\phantom{-}0.0297    &\phantom{-}0.0329    &\phantom{-}0.0075   &\phantom{-}0.0219   &-0.0185\\
    \phantom{-}0.0001    &\phantom{-}0.0106   & \phantom{-}0.0297   & \phantom{-}0.0319    &\phantom{-}0.0371    &\phantom{-}0.0071   & \phantom{-}0.0231   &-0.0166\\
   -0.0003   &\phantom{-}0.0118    &\phantom{-}0.0329    &\phantom{-}0.0371    &\phantom{-}0.0500    &\phantom{-}0.0076    &\phantom{-}0.0245   &-0.0164\\
    \phantom{-}0.0003   &\phantom{-} 0.0115   &\phantom{-} 0.0075   & \phantom{-}0.0071    &\phantom{-}0.0076    &\phantom{-}0.0065   & \phantom{-}0.0044  &-0.0115\\
   -0.0013   & \phantom{-}0.0110    &\phantom{-}0.0219   & \phantom{-}0.0231   & \phantom{-}0.0245    &\phantom{-}0.0044    &\phantom{-}0.0554   &-0.0140\\
    \phantom{-}0.0008   &-0.0141   &-0.0185   &-0.0166   &-0.0164  & -0.0115   &-0.0140    &\phantom{-}0.1271
\end{bmatrix}
\]
We have considered the same $ \mu = (1.0630; 1.0633; 1.0670; 1.0853; 1.0882; 1.0778; 1.0820; 1.1605) $ as given in \cite{Beck} and the minimal return is chosen as $ r_0 = 0.05 $ for the problem (\ref{MP}). We have carried out numerical experiments for this portfolio optimization problem with different vectors $ a $ in the upper level objective function of the problem (\ref{MP}). Using stopping criteria \eqref{SCB}, we got the following optimal solutions for the threshold value $ \varepsilon = 10^{-10} $.
\begin{enumerate}[(i)]
    \item If we want to invest the most in the first share, we can choose $ a = (1,0,0,0,0,0,0,0) $ and we get the optimal portfolio as
    \[ x^* = ( 0.6784, 0.0470, 0.0469, 0.0461, 0.0459, 0.0464, 0.0462, 0.0431).\]
    In this case, the upper optimum value is $ 5.910165e-02 $ and the lower optimum value is $ 1.759577e-03 $, which is obtained in iterations $ 11 $ (17.582327 sec).

    \item If we choose $ a = (\frac{1}{2},\frac{1}{2},0,0,0,0,0,0) $ {i.e.} we want to invest as much as possible in the first two shares, we get the solution
\[x^* = (0.6784, 0.0470, 0.0469, 0.0461, 0.0459, 0.0464, 0.0462, 0.0431)\]
with upper level optimal value $ 1.247960e-01  $, the lower level objective value $ 1.759577e-03 $ and it was done in $ 96 $ iterations (149.148218 sec ).

\item Whereas if we want a dispersed portfolio and choose $ a = (\frac{1}{8}, \frac{1}{8}, \frac{1}{8}, \frac{1}{8}, \frac{1}{8}, \frac{1}{8}, \frac{1}{8}, \frac{1}{8}) $, we get
\[x^* = (0.1250, 0.1250, 0.1250, 0.1250, 0.1250, 0.1250, 0.1250, 0.1250) \]
as the optimum portfolio. The upper level optimum value $ 8.073716e-20 $ and the lower level optimum value $  9.463094e-03 $ was obtained in $ 2 $ iterations (2.365044 sec).

\item If we choose $ a = (0,\frac{1}{2},0,0,0,0,0,0,\frac{1}{2}) $ {i.e.} we want to invest as much as possible in the second and last shares, we get the solution
\[x^*= (0.6784, 0.0470, 0.0469, 0.0461, 0.0459, 0.0464, 0.0462, 0.0431)\]
with upper level optimal value $ 1.247960e-01  $, the lower level objective value $ 1.759577e-03 $ and it was done in $ 96 $ iterations (149.148218 sec ).

\end{enumerate}
 \textbf{Remark:} Note that the solution in (ii) is same as in (i), which implies that the choice of $ a $ in the upper level function does not effect the solution. The reason behind this is the choice of the threshold value $ \varepsilon $, as we choose $ \varepsilon $ very small, we force the algorithm to focus on the solution of the lower level problem more than the upper level one. Hence leading to same solution as in (i) and (ii). But if we increase $ \varepsilon$ little bit and choose $ \varepsilon = 10^{-8} $, then we get $ x^* = (0.3627,0.3627,0.0469,0.0461,0.0459,0.0464,0.0462,0.0431)$ as the new solution for the case in (ii) which reflects the effect of $ a = (\frac{1}{2},\frac{1}{2},0,0,0,0,0,0) $ with optimal value $ f^* = 2.513309e-02 $, $ g^* = 5.528271e-03 $ in 2 iterations (7.850798 sec), while the solution $ x^* $ for the case (i) remains unchanged.

\bibliographystyle{jnsao}
\bibliography{DempeJNSAO}

\end{document}